\documentclass[preprint,12pt,nonatbib]{amsart}

\usepackage[pagewise]{lineno}
\usepackage[english]{babel}
\usepackage[latin1]{inputenc}
\usepackage{amsmath,amssymb,amsthm}
\usepackage[backend=biber,style=numeric,sorting=nyt,maxnames=99, doi=false,url=false]{biblatex}
\AtEveryBibitem{\clearfield{issn}}
\AtEveryBibitem{\clearfield{isbn}}

\DeclareMathOperator{\Val}{Val}
\DeclareMathOperator{\Eval}{ev}

\newtheorem{theorem}{Theorem}[section]
\newtheorem{lemma}[theorem]{Lemma}

\title{Interval Rings}

\author[Heinzer]{William Heinzer} 
\address[Heinzer]{Department of Mathematics, Purdue University, West Lafayette, Indiana 47907-1395 U.S.A., {\rm heinzer@purdue.edu}}

\author[Loper]{K.~Alan Loper}
\address[Loper]{Department of Mathematics, Ohio State University,  Columbus, OH 43210-1174, lopera@math.ohio-state.edu}

\author[Olberding]{Bruce Olberding} 
\address[Olberding]{Department of Mathematical Sciences, New Mexico State University, Las Cruces, NM 88003-8001 U.S.A., bruce@nmsu.edu}

\author[Toeniskoetter]{Matthew Toeniskoetter} 
\address[Toeniskoetter]{Department of Mathematics and Statistics, Oakland University, Rochester, MI 48309-4479, toeniskoetter@oakland.edu}

\thanks{ 
 \\The third author is  supported by NSF grant DMS-2231414. 
}

\begin{document}

\maketitle

\begin{abstract}
    Interval rings comprise a class of one-dimen\-sion\-al integrally closed local integral domains which are overrings of a two-dimen\-sion\-al regular local ring.  We construct them by intersecting carefully chosen valuation rings and prove that they have various nice properties.
    Our interest in these rings is that they represent a major stepping stone toward classifying all integrally closed overrings of a two-dimensional regular local ring.
\end{abstract}

\section{Introduction}

This paper is part of the authors' ongoing project to classify the integrally closed overrings of a two-dimensional regular local ring, where by \textit{overring}, we mean a ring between the base domain and its quotient field.
A first attempt at such a project might proceed by starting with a two-dimensional regular local ring $D$ and adding fractions, but if we add a finite number of fractions to $D$, we obtain rings that are quite similar to $D$.
However, above $D$, there are many integrally closed domains that do not fit into the standard categories: not Noetherian, not Krull, not Pr\"ufer.
Rather than a bottom-up approach of adjoining an infinite collection of fractions to $D$, we use a top-down approach of taking a collection of valuation domains above $D$ and intersecting them.

Let $V$ be an integral domain with quotient field $K$.
Then $V$ is a valuation ring if given any two nonzero elements $r,d \in V$, either $r/d$ or $d/r$ is in $V$.
A valuation ring $V$ is very close to being all of the quotient field $K$.
For example, the only domains properly between $V$ and $K$ are themselves valuation rings and are localizations of $V$.
If $V$ has finite Krull dimension, then there are only finitely many such intermediate rings.

On the other hand, if $D$ is a two-dimensional regular local ring with maximal ideal $\mathfrak{m} = (x,y)$ and quotient field $K$, then there are uncountably many valuation rings between $D$ and $K$.
This is an indication that valuation rings between $D$ and $K$ are abundant.
Another sign that valuation rings in the above setting exist in abundance is Krull's classical theorem stating that any integrally closed domain can be expressed as the intersection of the valuation rings between the integrally closed domain and its quotient field (see for example \cite[Corollary 5.22]{MR242802} or \cite[Theorem 10.4]{MR1011461}).

We will carefully choose collections of valuation rings to construct interesting examples by means of intersecting the valuation rings in these collections. 
The rings we consider here we call interval rings.
Their construction is remarkably simple.
For a closed interval $I = [a,b]$ of positive real numbers with $a<b$, the interval ring $T_I$ is obtained by intersecting a collection of two-dimensional valuation overrings of $D$ arising from the interval $I$.
We will see in a later paper that slightly generalized interval rings form the standard building blocks from which all integrally closed local overrings of $D$ that dominate $D$ can be constructed.

Among the properties we prove about interval rings $T_I$ are the following:
    \begin{itemize}
        \item $T_I$ is local.
        \item $T_I$ is one-dimensional but has many two-dimensional valuation overrings.
         \item $T_I$ is not Noetherian and not completely integrally closed.
        \item The complete integral closure ${T_I}^*$ of $T_I$ is Pr\"ufer, i.e.\ the valuation overrings of ${T_I}^*$ are precisely the localizations of ${T_I}^*$ at prime ideals.
        \item $T_I$ is vacant --- a technical term meaning that the valuation overrings are sparse in a certain topological sense
        which we explain later.
        \item $T_I$ is atomic, i.e., every non-zero, non-unit element of $T_I$ can be written as a finite product of irreducible elements.
        \item The natural topologies that can be placed on the space of valuation overrings of $T_I$ are induced by the ordinary real topology on the interval $I$.
    \end{itemize}

This is the third article in our recent project to classify integrally closed overrings of two-dimensional regular local rings.
In the first two articles \cite{MR4719886, 2406.10966}, the emphasis is somewhat different from what we present here.
However, we decided that, as far as possible, we would give the topic of interval rings a self-contained presentation in hopes that they will thereby be more readily understood.
In a future article, we will use interval rings as key components in a classification framework for integrally closed overrings of two-dimensional regular local rings.  

\section{Valuation rings}\label{SectionMonomialValuations}

We digress briefly to spell out some facts concerning valuation rings so that we will not have to make repeated digressions later.
Since we are working over a two-dimensional regular ring instead of a polynomial ring, the construction of our valuation rings, which follows well-known traditional methods in the polynomial ring setting, needs more care.

Fix, once and for all, a two-dimensional regular local ring $D$ with quotient field $K$ and a regular system of parameters $x,y$, so that its maximal ideal is $\mathfrak{m} = (x, y)$.
Following Granja \cite[Lemma 7]{MR2289617}, every nonzero element of $D$ can be written in the form $u_1\tau_1 + \cdots + u_n\tau_n$, where the $u_i$ are units in $D$, the $\tau_i$ are distinct pure monomials in $x$ and $y$, and the $\tau_i$ do not divide one another, i.e.\ $\tau_i D \not \subseteq \tau_j D$ for $i \ne j$.
Moreover, for this chosen regular system of parameters, these properties uniquely determine the $\tau_i$ (but not the $u_i$).

This representation yields the \textit{monomial valuations} on $K$, with respect to the regular system of parameters $\mathfrak{m} = (x, y)$, by assigning positive real values to the parameters $x$ and $y$.
The corresponding valuation rings are the \textit{monomial valuation rings}.
Since $x$ and $y$ will have positive real values, we may normalize so that $x$ has value~$1$.
Then for a positive real number $r$, we define a map $v_r :D \rightarrow {\mathbb{R}} \cup \{\infty\}$ by $v_r(0) = \infty$, $v_r(x^p y^q) = p + q r$, and $$v_r (u_1\tau_1 + \cdots + u_n \tau_n) = \min \{v_r (\tau_i):1 \leq i \leq n\},$$ where $u_1\tau_1+\cdots + u_n\tau_n$ is as above.
The map $v_r$ extends to a map on the quotient field $K$ of $D$ by setting $$v_r (f/g) = v_r (f) - v_r (g)$$ for all $f,g \in D$ with $g \ne 0$.
With this extension, it follows from \cite[Lemma 8]{MR2289617} that $v_r$ is a valuation on $K$, i.e.\ $0$ is the element of $K$ with value $\infty$ and for all $f,g \in K$, 
$$v_r(fg) = v_r(f)+v_r(g) {\mbox{ and }} v_r(f+g)\geq \min\{v_r(f),v_r(g)\},$$ with equality in the last case whenever $v_r(f) \ne v_r(g)$.
The subset of $K$, $$V_r = \{f \in K:v_r(f) \geq 0\}$$ is a ring with quotient field $K$ and is the valuation ring associated to $v_r$. 
The class of monomial valuations is central to our work, and we shall consider the topological properties of this collection of valuation rings.

Recall that a \textit{DVR} is a Noetherian valuation domain that is not a field, necessarily one-dimensional.
A monomial valuation ring $V_r$ is a DVR if and only if $v_r (y)$ is rational.

Let $\Val (D)$ be the set of valuation overrings of $D$ that dominate $D$.
We define three topologies on this space of valuation rings.

\begin{itemize} 
\item The \textit{Zariski topology} has a basis, and the \textit{inverse topology} has a closed subbasis, given by prescribing a finite set of elements with non-negative value,
\[  \{V \in \Val (D) : F \subseteq V\} \quad \text{where $F$ is a finite subset of $K$.} \]
\item The \textit{patch topology} has a basis consisting of clopen sets given by prescribing a finite set of elements with non-negative value and a finite set of elements with positive value,
\[  \{V \in \Val (D) : F \subseteq V , G \subseteq \mathfrak{m}_V \} \quad \text{where $F,G$ are finite subsets of $K$.} \]
\end{itemize}

The patch topology is the common refinement of the Zariski topology and the inverse topology, and it is Hausdorff.
Let $X \subseteq \Val (D)$ be a patch-closed set of valuation rings.
By \cite[Remark 2.2]{MR3105748}, the closures of $X$ in the Zariski and inverse topologies, respectively, are
    $$X^{\downarrow} = \{ V \mid V \subseteq W \text{ for some } W \in X \}, \:\: X^{\uparrow} = \{ V \mid V \supseteq W \text{ for some } W \in X \}.$$

\section{Evaluation Functions}\label{SectionEvaluation}

For a nonzero element $\alpha \in K$, consider its \textit{evaluation function} $\Eval_\alpha$, which is the function defined by evaluating the element $\alpha$ on each monomial valuation,
\begin{align*}
    \Eval_\alpha : \mathbb{R}^+ &\longrightarrow \mathbb{R}:
    \lambda \longmapsto v_\lambda (\alpha)
\end{align*}

The purpose of this section is to completely classify these evaluation functions as the real-valued functions satisfying the following conditions:
\begin{enumerate}
    \item $\Eval_\alpha$ is a continuous piecewise-linear function,
    \item $\Eval_\alpha$ has finitely many points where its slope changes, each of which is a rational number, and
    \item $\Eval_\alpha$'s linear components have integer slopes and $y$-intercepts.
\end{enumerate}
We show that $\Eval_\alpha$ satisfies these conditions, and that if $e : \mathbb{R}^+ \rightarrow \mathbb{R}$ is a function satisfying these conditions, there exists nonzero $\alpha \in K$ such that $e = \Eval_\alpha$.

We start with the following observations.

\begin{itemize}
\item $\Eval_x$ is the constant function, $\Eval_x (\lambda) = 1$ for all $\lambda \in \mathbb{R}^+$, and $\Eval_y$ is the identity function, $\Eval_y (\lambda) = \lambda$ for all $\lambda \in \mathbb{R}^+$.
\item If $\alpha = y^a x^b$ (i.e.\ $\alpha$ is a monomial), then $\Eval_\alpha(\lambda) = a \lambda + b$.

\item If $\Eval_f$ and $\Eval_g$ have no open intervals on which they agree, then $$\Eval_{f + g} (\lambda) = \min \{ \Eval_f (\lambda), \Eval_g (\lambda) \}.$$

\item If $\alpha = u_1\tau_1 + \cdots + u_n\tau_n \in D$, as in Granja's decomposition in Section~\ref{SectionMonomialValuations}, then each $\Eval_{\tau_i}$ is represented by an expression of the form $a_i + b_i \lambda$, corresponding to $n$ distinct lines in the plane.
Hence
    \begin{eqnarray*}
    \Eval_\alpha (\lambda) & = & \min \{ \Eval_{\tau_i} (\lambda) \mid 1 \le i \le n \} \\
    &= & \min \{ a_i + b_i \lambda \mid 1 \le i \le n \}.
    \end{eqnarray*}

\item If $\alpha = \beta / \gamma \in K$, where $\beta, \gamma \in D$, then $\Eval_\alpha = \Eval_\beta - \Eval_\gamma$.
Thus $\Eval_\alpha$ is also a continuous piecewise linear function with finitely many pieces.

\item If the slope of $\Eval_\alpha (\lambda)$ changes at a rational number $p / q$, written in lowest form, then the slope changes by an integer multiple of $q$.
To see this, say $\Eval_\alpha$ is $a \lambda + b$ near the left of $p / q$ and $c \lambda + d$ near the right of $p / q$, where $a, b, c, d \in \mathbb{Z}$.
By continuity, $a (p / q) + b = c (p / q) + d$, so $a p + b q = c p + d q$.
Then $(a - c) p = (d - b) q$, so since $p$ and $q$ are relatively prime, $q$ divides $(a - c)$.

\end{itemize}

Thus the evaluation function satisfies the conditions listed above.

Conversely, let $e : \mathbb{R}^+ \rightarrow \mathbb{R}$ be a function satisfying these conditions.
We use induction on the number of points $n$ where the slope of $e$ changes.

In the base case of $n = 0$, the function $e$ has no changes of slope, so $e (\lambda) = a \lambda + b$ for some $a, b \in \mathbb{Z}$.
Then $e = \Eval_f$, where $f = y^a x^b$.

We also show the case $n = 1$, where the function $e$ has one change of slope, say at the rational number $p / q$ written in lowest form.
Write $e (\lambda) = a \lambda + b$ on $(0, p / q]$ and $e (\lambda) = c \lambda + d$ on $(p / q, \infty)$.
Either $a > c$ or $a < c$.
By replacing $e$ with $- e$ in the case where $a < c$ and then taking reciprocals, we may assume without loss of generality that $a > c$.
Then consider the elements $g = y^a x^b$ and $h = y^c x^d$, so $\Eval_g (\lambda) = a \lambda + b$ and $\Eval_h (\lambda) = c \lambda + d$.
Since $a > c$ and $\Eval_g (p / q) = \Eval_h (p / q)$, it follows that $\Eval_g (\lambda) < \Eval_h (\lambda)$ on $(0, p / q)$ and $\Eval_g (\lambda) > \Eval_h (\lambda)$ on $(p / q, \infty)$.
Thus $\Eval_{g + h} = e$, which completes the $n = 1$ case.

In the inductive step, suppose such an element exists for every function with $n$ changes in slope, where $n \ge 1$, and suppose $e$ has $n + 1$ changes in slope.
Suppose the rightmost input where $e$ has a change in slope is $p / q$, written in lowest form, and let $e'$ be the function with this change in slope removed.
That is, given that $e (\lambda) = a \lambda + b$ in the interval $(p / q - \epsilon, p / q]$ for some $0 < \epsilon < p / q$, then $e' (\lambda) = a \lambda + b$ on $[p / q, \infty)$.
Write $e (\lambda) = c \lambda + d$ on $[p / q, \infty)$.
As in the $n = 1$ case, we may assume without loss of generality that $a > c$.
By the induction hypothesis, there exists nonzero $g \in K$ with $\Eval_g = e'$.

It is natural to consider the element $g + y^c x^d$, which agrees with $e$ on the interval $(p / q - \epsilon, \infty)$, but it may be the case that $\Eval_{y^c x^d} (\lambda) \le \Eval_g (\lambda)$ in some segment of $(0, p / q)$.
To ensure this does not occur, choose a sufficiently large integer $M$ such that $$(c - M q) (\lambda - p / q) + e (p / q) \ge e (\lambda)$$ on $(0, p / q)$.
Then consider the piecewise-linear function $e''$, where
    $$e'' (\lambda) = \begin{cases}
        c \lambda + d & \text{if } \lambda \ge p / q,\\
        (c - M q) \lambda + (d + M p) & \text{if } 0 < \lambda \le p / q.
    \end{cases}$$
Then, for all $0 < \lambda \le p / q$,
\begin{align*}
    e' (\lambda)
        &= e (\lambda) \\
        &\le (c - M q) (\lambda - p / q) + e (p / q) \\
        &= (c - M q) \lambda - (c p / q - M p) + (c p / q + d)  \\
        &= (c - M q) \lambda + (d + M p) \\
        &= e'' (\lambda).
\end{align*}
Thus $e (\lambda) = \min \{ e' (\lambda), e'' (\lambda) \}$ for all $\lambda \in \mathbb{R}^+$.
By the $n = 1$ case, there exists nonzero $h \in K$ such that $e'' = \Eval_h$.
Then $e = \Eval_{g + h}$.

\section{Krull-constructed rings}\label{SectionKCR}

Now that the monomial valuation rings are defined, we turn our focus to the rings associated to them.
The valuations $v_r$ are defined for positive real numbers $r$, with associated valuation rings $V_r$.
For a rational value of $r$, the residue field of $V_r$ is a simple transcendental extension $k (t)$ of the residue field $k$ of $D$.
In this case, $V_r$ has a ``projective line'' of rank two valuation subrings associated to the function field $k (t)$, and these valuation rings, in a strong sense, are ``infinitesimally close'' to $V_r$.
The intersection of any proper subcollection of these rank two valuation rings is a Pr\"ufer domain whose maximal ideals correspond to the valuation rings in the subcollection,
but the intersection of all of them is a local ring, which we call the \textit{Krull-constructed ring} associated to $V_r$ because it originates with a construction due to Krull \cite[p. 670]{MR1545646}, as described by Seidenberg \cite[p. 509]{MR54571}.
It is a pullback of $V_r$, sharing a maximal ideal with $V_r$ but having fewer units.

Formally, for a rational number $r$, let $\psi_r$ be the quotient map from $V_r$ onto its residue field $k(t)$.
Let $\phi(t)$ be an irreducible polynomial in the subring $k[t]$ of $k(t)$.
Then the localization $k[t]_{(\phi(t))}$ is a DVR.  
The inverse image $$W_{\phi ,r} =  \psi_r^{-1}(k[t]_{(\phi(t))}),$$ which is the pullback of the DVR $k[t]_{(\phi(t))}$ in $V_r$, is then a rank two valuation overring of $D$ with maximal ideal generated by any lift of $\phi(t)$.
In addition, the final rank two valuation subring of $V_r$ is the inverse image of the ``point at infinity,'' $$W_{1/t, r} = \psi_r^{-1} (k[1/t]_{(1/t)}),$$ whose maximal ideal is generated by any representative of $1/t$.
Note that the unique height one prime ideal of $W_{\phi, r}$ is equal to the maximal ideal of $V_r$.

The intersection of these DVRs of $k(t)$, as $\phi(t)$ ranges over the irreducible polynomials in $k[t]$ together with the point at infinity, is equal to $k$.
The intersection of their pullbacks, the rank two valuation subrings of $V_r$, is what we call the \textit{Krull-constructed ring (KCR)} associated to $V_r$, say $U_r$.
Then $U_r = \psi_r^{-1} (k)$, where $k$ is viewed as a subfield of the residue field $k(t)$ of $V_r$, is a local ring whose maximal ideal is equal to the maximal ideal of $V_r$.

Although $U_r$ has the same non-units as $V_r$, it has many fewer units.
The unit group of $U_r$ is the intersection of the unit groups of the rank two valuation subrings of $V_r$.

A KCR can be thought of as binding together valuation rings that are ``close'' together, which causes their maximal ideals to ``collapse'' together and produce a local ring.
In turn, we shall bind together KCRs that are ``close'' to one another, also, to produce local rings.

\section{The construction of interval rings}\label{SectionConstruction}

We continue to work with the two-dimensional regular local ring $D$ with maximal ideal $\mathfrak{m} = (x, y)$.
The interval ring will be an overring of $D$, by which we mean that it lies between $D$ and its quotient field.
We construct it by choosing an interval $I = [a,b]$ of positive real numbers with $a < b$ and defining the interval ring $T_I$ for $I$ to be the intersection of all the valuation rings $V$ such that $D \subseteq V \subseteq  V_r$ for some $r \in I$, where, as usual, $V_r$ is the valuation ring for the valuation $v_r$ defined in Section~\ref{SectionMonomialValuations}.

For the rational values $r \in I$, the intersection of the rank two valuation rings $V$ such that $D \subseteq V \subseteq V_r$ is precisely the Krull-constructed ring $U_r$ as in Section~\ref{SectionKCR}.
In fact, the rank one valuation rings $V_r$, for $r \in I$, are redundant in the construction, and $T_I$ can be defined completely as the intersection of the $U_r$, for rational $r \in I$.
To see this, let $f$ be a nonzero element in the intersection of the $U_r$, for rational $r \in I$, and consider its evaluation map $\Eval_f : I \rightarrow \mathbb{R}$, which is a continuous function.
Since $f \in U_r$ for each rational $r \in I$, $\Eval_f (r) \ge 0$ for each rational $r \in I$.
By continuity, $\Eval_f (r) \ge 0$ for all $r \in I$, so $f \in V_r$ for all $r \in I$.
Since also $f \in U_r$ for each rational $r \in I$, we conclude $f \in T_I$ as in the original definition.

As a consequence of the definition of $T_I$, the intersection $M_I$ of the maximal ideals of the $V_r$, for $r \in I$, is contained in $T_I$.
Moreover, $M_I$ is equal to the intersection of the maximal ideals of the $V_r$, for rational $r \in I$.
Ultimately, we will show in Section~\ref{SectionLocal} that $M_I$ is the unique maximal ideal of $T_I$, and we will show in Section~\ref{SectionOneDimensional} that $M_I$ is the only nonzero prime ideal in $T_I$.

Even for a fixed interval $I$, the ring $T_I$ constructed in this manner depends on the choice of regular system of parameters $\mathfrak{m} = (x, y)$.
With a different choice, the ring $T_I$ constructed may be different, possibly non-isomorphic.

\section{The internal rank two valuation rings}\label{SectionInternal}

Let $r$ be a positive rational number, say $r = p / q$ in reduced form, and write
the residue field of $V_r$ as   $k (t)$, where $t$ is the image of $y^q / x^p$ and $k$ is the residue field of $D$.
As in Section~\ref{SectionKCR}, every rank two valuation subring of $V_r$ arises from a DVR of $k(t)$, and we now give an alternate realization of $W_{t, r}$ and $W_{1/t, r}$ as limits.
In the notation of Section~\ref{SectionMonomialValuations}, define the sets,
\begin{align*}
    {W_r}^+ &=
        \{ f \in K \mid \Eval_f (s) \ge 0 \text{ for all } s \in (r, r + \delta) \text{ for some } \delta > 0 \}, \\
    {M_r}^+ &=
        \{ f \in K \mid \Eval_f (s) > 0 \text{ for all } s \in (r, r + \delta) \text{ for some } \delta > 0 \},
\end{align*}
and symmetrically,
\begin{align*}
    {W_r}^- &=
        \{ f \in K \mid \Eval_f (s) \ge 0 \text{ for all } s \in (r - \delta, r) \text{ for some } \delta > 0 \}, \\
    {M_r}^- &=
        \{ f \in K \mid \Eval_f (s) > 0 \text{ for all } s \in (r - \delta, r) \text{ for some } \delta > 0 \}.
\end{align*}
We claim that ${W_r}^+ = W_{t, r}$ and ${W_r}^- = W_{1/t, r}$, with corresponding maximal ideals ${M_r}^+$ and ${M_r}^-$.
We prove this for ${W_r}^+$, with the ${W_r}^-$ case similar.

To verify that ${W_r}^+ = W_{t, r}$, we first show that ${W_r}^+$ is a valuation ring, and then that ${W_r}^+ \subseteq V_r$, and that the preimage $y^q / x^p$ of $t$ is a nonunit in ${W_r}^+$.
As in Section~\ref{SectionKCR}, this identifies ${W_r}^+$ as the pullback of $k [t]_{(t)}$ in $V_r$.
This will follow from the nature of the evaluation function $\Eval_f$ as a continuous piecewise linear function, with finitely many pieces defined on rational intervals.

Let $f \in {W_r}^+$, so that $\Eval_f$ is nonnegative on the interval $(r, r + \delta)$ for some $\delta > 0$.
By continuity of $\Eval_f$, it follows that $\Eval_f (r) \ge 0$, and so $f \in V_r$.
Thus ${W_r}^+ \subseteq V_r$.

The set ${W_r}^+$ is a subring of $K$, since for two elements in ${W_r}^+$ we may choose a common interval $(r, r + \delta)$ on which both evaluation functions are nonnegative and use the valuation inequalities on this interval.
To show that ${W_r}^+$ is a valuation ring, let $f \in K$ be nonzero.
If $\Eval_f (r) > 0$, i.e.\ if $f \in M_r$, then by continuity of $\Eval_f$, $f \in {M_r}^+$.
Similarly, if $\Eval_f (r) < 0$, i.e.\ if $\frac{1}{f} \in M_r$, then $\frac{1}{f} \in {M_r}^+$.
If $\Eval_f (r) = 0$, then since $\Eval_f$ is piecewise linear with finitely many pieces, there exists $\delta > 0$ such that $\Eval_f$ is linear on $[r, r + \delta]$.
Thus the sign of $\Eval_f$ is constant on $(r, r + \delta)$.
If the sign of $\Eval_f$ is positive on $(r, r + \delta)$, then $f \in {M_r}^+$; if negative, $\frac{1}{f} \in {M_r}^+$; if zero, then $f, \frac{1}{f} \in {W_r}^+$.
Since either $f$ or $\frac{1}{f}$ is in ${W_r}^+$ in any case, it follows ${W_r}^+$ is a valuation ring contained in $V_r$.
It is readily checked that ${M_r}^+$ is precisely the set of nonunits of the valuation ring ${W_r}^+$, so ${M_r}^+$ is the maximal ideal of ${W_r}^+$.

Finally, observe that, for all $h > 0$,
    \begin{eqnarray*}
    \Eval_{y^q/x^p} (r + h) 
        &= & v_{p/q + h} (y^q / x^p) \\
        &= &  q \, v_{p/q + h} (y) - p \, v_{p/q + h} (x) \\
        & = & q (p/q + h) - p
        \:\: = \:\; q h,
        \end{eqnarray*}
so that $\Eval_{y^q/x^p} (r + h) > 0$ for all $h > 0$.
Thus $y^q/x^p \in {M_r}^+$, distinguishing ${W_r}^+$ as the rank $2$ valuation subring $W_{t, r}$ of $V_r$.

For a given interval $I = [a, b]$ of positive real numbers with $a < b$,
we call the rank $2$ valuation rings ${W_r}^+$ and ${W_r}^-$, where $r$ is a rational number in $(a, b)$, the \textbf{internal} rank $2$ valuation rings of the interval.
In addition, if $a$ is rational, ${W_a}^+$ is an internal rank $2$ valuation ring, and if $b$ is rational, ${W_b}^-$ is an internal rank $2$ valuation ring.
We call all other rank $2$ valuation rings the \textbf{external} rank $2$ valuation rings of the interval.

\section{Interval rings have few valuation overrings}\label{ValuationOverringsSection}

Although an interval ring has uncountably many one-dimensional valuation overrings (namely, the $V_r$, for $r \in I$), there is little diversity in this uncountable collection.
They are each built with a simple algorithm based on a single positive real number, and in that sense, they are almost identical to one another; they only differ in terms of which real number they are based on, and whether that number is rational or irrational.

Again, we choose two real numbers $0 < a < b$, let $I = [a,b]$ and build the interval ring $T_I$ as in Section~\ref{SectionConstruction}.
We first prove that the domains $V_r$ are the only one-dimensional valuation overrings of $T_I$.
This is a powerful result that will make many later results easier to prove.  

Suppose to the contrary that $W$ is a one-dimensional valuation overring of $T_I$ which is not one of the $V_r$.
We claim that both $x$ and $y$ must lie in the maximal ideal of $W$.
Recall that $W$ is a valuation overring of the two-dimensional regular local ring $D$.
If $W$ does not contain both $x$ and $y$ as non-units, then the center of $W$ on $D$ is a height-one prime.
Since $D$ is a UFD, this prime is generated by an irreducible element $z \in D$, and the corresponding localization is $D_{(z)}$.
We show that no such localization can contain $T_I$.
Consider our interval $I = [a,b]$. For any $r \in I$ and any positive
integer $n$ we have $v_r(x^n) = n$ and $v_r(y^n) = rn$.  
Since the evaluation function $\Eval_z$ is continuous, it attains its maximum value on the closed and bounded interval $I$.
It follows then that for sufficiently large $n$, both $x^n/z$ and $y^n/z$ are in the intersection of the maximal ideals of the $V_r$ and hence, as noted in Section~\ref{SectionConstruction}, in the ring $T_I$.
Since $z$ cannot divide both $x^n$ and $y^n$, at least one of $x^n/z$ and $y^n/z$ is not in $D_{(z)}$.
Thus, the localization of $D$ at $(z)$ is not an overring of $T_I$.  

Let $w$ be a real-valued valuation associated with $W$, so that $w (x) > 0$ and $w (y) > 0$, and normalize $w$ so that $w(x) = 1$.
We consider the cases where $w (y)$ lies in or out of the interval $I$.
In each case we aim to find an element which is in $T_I$ but not in $W$, which would contradict the  assumption that $W$ is an overring of $T_I$.

In the first case, $w(y) > b$.
Choose a rational number $m/n$ such that $ b < m/n < w(y)$ and consider the element $x^m / y^n$. Then,

\[w \left( \frac{x^m}{y^n} \right) = m - nw(y) < m-n(m/n) = 0,\]
so $w(x^m / y^n) < 0$ and $x^m / y^n \notin W$.
On the other hand, since we have $b < m/n$, it follows that $m-nb > 0$. For each real number $r$ with $a \leq r \leq b$,
\[  v_r\left( \frac{x^m}{y^n} \right) = m-nr > m-nb > 0, \]
which implies that $x^m / y^n$ is in the intersection of the maximal ideals of $V_r$ and hence in the ring $T_I$.

In the second case, $0 < w(y) < a$.
An argument parallel to the one above constructs an element in $T_I$ that is not in $W$.  

In the final case, $w(y) \in I$.
Let $r = w(y)$, so the two valuations $w$ and $v_r$ agree at $x$ and $y$.
Clearly, they also agree on monomials of the form $u x^n y^m$, where $u \in D$ is a unit.
But by assumption, there must be an element $h \in D$ for which $w$ and $v_r$ do not agree.  
Recall that the monomial valuation $v_r$, having fixed $v_r (x) = 1$ and $v_r (y) = r$, assigns each element of $D$ the unique minimal possible value.
Since $w$ and $v_r$ disagree at $h$, then we must have $v_r (h) < w (h)$.
As in Section~\ref{SectionMonomialValuations}, express $h$ using Granja's decomposition as a sum of products of units in $D$ and monomials in the $x$ and $y$, say $$h = k_1 + \cdots + k_m + g,$$ where $m \ge 1$ and $g$ has minimal $v_r$ value among the terms, i.e.\ $v_r (h) = v_r (g) = w (g)$.
Denote $$f = h - g = k_1 + \cdots + k_m.$$
Since $w (h) > w (g)$, it follows that $w (h) = w (f + g) > \min \{ w (f), w (g) \}$.
Thus $w (f) = w (g)$, so in particular $w (f) = w (g) = v_r (h) = v_r (f) = v_r (g)$.

Let $\alpha = \frac{f}{h}$.
Since every monomial of $f$ is also a monomial of $h$, it follows that $v_p (f) \ge v_p (h)$ for all $p > 0$.
Thus $v_p (\alpha) \ge 0$ for all $p > 0$.
Then for all $n > 0$ and all $p > 0$, since $v_p (x) = 1$ and $v_p (\alpha^n) \ge 0$, it follows that $v_p (x \alpha^n) > 0$.
Thus $x \alpha^n \in T_I$ for all $n > 0$.

However, $w (h) > w (f)$, so $w (\alpha) < 0$.
Then there exists an integer $n$ such that $w (x \alpha^n) < 0$.
Hence the element $x \alpha^n$ is in $T_I$ but not in $W$, so we conclude $W$ cannot be an overring of $T_I$.\footnote{
Our original argument used an assumption that the residue field has more than two elements to construct an element $\alpha$ such that $v_p (\alpha) = 0$ for all $p \ge 0$.
We would like to thank the referee for suggesting this simpler and more general argument.
}

We finish this section by concluding that the valuation rings used to construct $T_I$ are all the valuation overrings of $T_I$.
Namely, since every two-dimensional valuation overring of $T_I$ is contained in a one-dimensional valuation overring of $T_I$, it follows that every two-dimensional valuation overring of $T_I$ is contained in one of the $V_r$, and hence is one of the valuation rings we used in Section~\ref{SectionConstruction} to construct $T_I$.
Thus the valuation overrings of $T_I$ are precisely,
\begin{itemize}
    \item the rational monomial valuation rings $V_r$,
    \item the irrational monomial valuation rings $V_r$,
    \item the internal rank $2$ valuation overrings associated to the $V_r$, and
    \item the external rank $2$ valuation overrings associated to the $V_r$.
\end{itemize}

\section{Interval rings are local}\label{SectionLocal}

In this section, we show interval rings are local.
Since an interval ring $T_I$ is defined as the intersection of valuation rings, the essence of the proof is to show that every non-unit of $T_I$ is a non-unit of the valuation rings $V_r$, $r \in I$.
Once this is shown, we use the observation from Section~\ref{SectionConstruction} that the intersection $M_I$ of the maximal ideals of the $V_r$ is contained in $T_I$ to conclude that the non-units in $T_I$ form an ideal and hence $T_I$ is local with maximal ideal $M_I$.

Let $r \in I$ be rational, and recall from Section~\ref{SectionKCR} that the group of units of $U_r$ is the intersection of the groups of units of the rank $2$ valuation rings associated to $V_r$.
In particular, if $f \in U_r$ and $\Eval_f (r) = 0$, then $f$ is a unit in each rank $2$ valuation ring associated to $V_r$.

Assume the notation of the preceding sections, let $d$ be a nonzero non-unit in $T_I$, and suppose to the contrary that $d$ is a unit in $V_{r_0}$ for some $r_0 \in I$.
Consider the evaluation function $\Eval_d : I \rightarrow \mathbb{R}$.
Since $d$ is a unit in $V_{r_0}$, we have $\Eval_d(r_0)=0$.
If $\Eval_d (q) = 0$ for all rational $q \in I$, then the observation above implies that $d$ is a unit in $U_q$ for all rational $q \in I$, so since $T_I$ is the intersection of the $U_q$, $d$ would be a unit in $T_I$, contradicting the choice of $d$ as a non-unit.
Therefore, $\Eval_d$ must take some positive values in the interval $(a, b)$.
Since $\Eval_d$ is a continuous piecewise linear function with finitely many pieces and rational breakpoints, and since it has a zero in $I$, there exists a rational number $r \in I$ with $\Eval_d(r)=0$ and a $\delta>0$ such that $\Eval_d(s)>0$ on either $(r,r+\delta)\cap I$ or $(r-\delta,r)\cap I$.
Therefore, $d \in V_r$ is a unit, but either $d \in {M_r}^+$ or $d \in {M_r}^-$ as in Section~\ref{SectionInternal}.
Since ${W_r}^+$ and ${W_r}^-$ are rank $2$ valuation subrings of $V_r$, it follows that $d$ is a non-unit in some rank $2$ valuation subring of $V_r$.
This contradiction shows that every nonzero non-unit of $T_I$ lies in $M_I$; conversely, since $T_I \subseteq V_r$ for each $r \in I$, every element of $M_I$ is a non-unit of $T_I$.
Thus the non-units of $T_I$ are exactly $M_I$, and $T_I$ is local with maximal ideal $M_I$.

\section{Interval rings are one-dimensional}\label{SectionOneDimensional}

The goal of this section is to show that interval rings are one-dimensional.
As above, let $I=[a,b]$ be a closed interval of positive real numbers with $a< b$, and let $T_I$ be the corresponding interval ring with maximal ideal $M_I$.
We show that $M_I$ is the radical of the principal ideal $\alpha T_I$ for any nonzero non-unit $\alpha \in M_I$, from which the one-dimensionality of $T_I$ follows.

Let $\alpha,\beta \in M_I$ be nonzero, and consider the evaluation functions $\Eval_\alpha, \Eval_\beta : I \rightarrow \mathbb{R}^+$ as in Section~\ref{SectionEvaluation}.
Since $I$ is compact, $\Eval_\alpha, \Eval_\beta$ have minimal and maximal positive values, so there exists some positive integer $k$ such that $k \cdot \min \Eval_\beta (I) > \max \Eval_\alpha (I)$.
Then for all $\lambda \in I$,
    $$v_\lambda \left( \frac{\beta^k}{\alpha} \right) = k \cdot v_\lambda (\beta) - v_\lambda (\alpha) \ge k \cdot \min \Eval_\beta (I) - \max \Eval_\alpha (I) > 0,$$
proving that $\frac{\beta^k}{\alpha} \in M_I$ and thus $\beta^k \in \alpha T_I$.
This shows that $M_I$ is the radical of $\alpha T_I$, and so $T_I$ is one-dimensional.

\section{The real number topology on valuation overrings}\label{SectionRealTopology}

Let $X = \Val (T_I)$, the set of valuation rings that birationally dominate the local ring $T_I$. 
We shall prove that by ``gluing together'' each rank $1$ valuation ring of $X$ with each of its valuation subrings, the quotient space becomes homeomorphic to the real interval.

Consider the space $X$ with the patch topology as defined in Section~\ref{SectionMonomialValuations}.
Since $T_I$ is one-dimensional, $X$ is the set of all valuation overrings of $T_I$ except for the quotient field.
Since this is the set of valuation overrings of $D$ that dominate the local ring $T_I$, this set is patch closed in the space of all valuation overrings of $D$.
The latter set is compact in the patch topology (it is a spectral space---see the discussion in \cite{MR4394424} for a range of references regarding this fact---and the patch topology on a spectral space is compact \cite[Theorem~1]{MR251026}), and so the closed subset $X$ is also compact in the patch topology.

Let $\sim$ be the equivalence relation such that for all $V, W \in X$, we have $V \sim W$ if $V$ and $W$ have the same rank $1$ overring.
Thus for any $r \in I$ and any of the two-dimensional valuation subrings $W_1$ and $W_2$ of $V_r$ described in Section~\ref{SectionKCR}, we have $V_r \sim W_1 \sim W_2$.

Define $\tilde{X} = X / \mathord{\sim}$ to be the quotient space, which is then also compact.
Consider the map $\phi : X \rightarrow I$, where for each $W \in X$, $\phi (W) = r$ if $V_r$ is the one-dimensional overring of $W$.
This map induces a bijection $\tilde{\phi} : \tilde{X} \rightarrow I$.
Since $\tilde{\phi}$ is a bijection from a compact space to a Hausdorff space, to see that $\tilde{\phi}$ is a homeomorphism, it suffices to show that it is continuous.
Thus it suffices to show that the preimage of any basic closed set $[\alpha, \beta] \subseteq I$ is closed.

Let $J$ be a closed subinterval $J = [\alpha, \beta] \subseteq I$ with $\alpha < \beta$.
Then $\phi^{-1}(J)$ is the collection of valuation rings in the definition of an interval ring that
are intersected to construct the interval ring $T_J$.
As in Section~\ref{ValuationOverringsSection}, this is the set of all valuation overrings of $T_J$ except the quotient field, so $\phi^{-1} (J) = \Val (T_J)$, which is closed in the patch topology.
We conclude that the map $\phi$ is continuous, which implies the quotient map $\tilde{\phi}:\tilde{X} \rightarrow I$ is continuous and hence a homeomorphism.

As discussed in Section~\ref{SectionMonomialValuations}, for a given subset of $X$ the difference between the closures in the patch and Zariski topologies involves only adding or deleting subrings or overrings.
Hence the quotient spaces for the patch and Zariski topologies are the same, so the quotient map $\tilde{\phi}$ is a homeomorphism with respect to the Zariski topology also.

The space $X$ is connected in the Zariski topology.
This can be shown directly from the fact that $X / \sim$ is a connected space and the fibers of $\sim$ are connected. This fact can also be obtained as a corollary of {\cite[Theorem~3.8]{MR4719886} or \cite[Theorem~1.2]{2406.10966}.}

\section{Almost Integral Elements}

To show in the next sections that interval rings are vacant and that the complete integral closure of an interval ring is Pr\"ufer, we use the existence of classes of explicitly-constructed almost integral elements.
Recall that, if $A \subseteq B$ is an extension of domains, then an element $x \in B$ is \textit{almost integral over $A$} if $A [x]$ is contained in a finitely generated $A$-module, i.e.\ if there exists nonzero $a \in A$ such that $a x^n \in A$ for all $n \ge 1$.

The set of almost integral elements over $T_I$ in $K$ is called its \textit{complete integral closure} ${T_I}^*$.
Since $T_I$ is a one-dimensional local domain, it satisfies the conditions of \cite[Proposition 4]{MR220714}, so the complete integral closure of $T_I$ is the intersection of the rank one valuation rings that contain $T_I$.
That is,
    $${T_I}^* = \bigcap_{\lambda \in I} V_\lambda = \{ f \in K \mid v_{\lambda} (f) \ge 0 \text{ for all } \lambda \in I \}.$$

For any rational number $r = p / q \in [a, b]$, written in lowest terms, consider the ratio $y^q / x^p$, which is residually transcendental for the valuation ring $V_r$.
The following lemma shows how to create almost integral elements from $y^q / x^p$ while controlling where the element has positive value.

\begin{lemma}\label{ElementLemma}
    Let $r \in [a, b]$ be rational, and let $W$, $W'$ be any pair of distinct two-dimensional valuation overrings of $D$ contained in $V_r$, neither of which is ${W_r}^-$ or ${W_r}^+$.
    Then for each of the following conditions, there exists an element $f$, residually transcendental only for $V_r$ and almost integral over $T_I$, that satisfies the condition and is such that $f$ is a unit in all two-dimensional valuation rings contained in $V_r$ not mentioned.
    \begin{enumerate}
        \item\label{LeftElement} $v_\lambda (f) > 0$ for $\lambda < r$ (so $f \in {M_r}^-$), $v_\lambda (f) = 0$ for $\lambda \ge r$, and $\frac{1}{f}$ is a non-unit in $W$.
        \item\label{RightElement} $v_\lambda (f) = 0$ for $\lambda \le r$, $v_\lambda (f) > 0$ for $\lambda > r$ (so $f \in {M_r}^+$), and $\frac{1}{f}$ is a non-unit in $W$.
        \item\label{BothElement} $v_r (f) = 0$, $v_\lambda (f) > 0$ for $\lambda \ne r$ (so $f \in {M_r}^- \cap {M_r}^+$), and $\frac{1}{f}$ is a non-unit in $W$.
        \item\label{SideElement} $v_\lambda (f) = 0$ for all $\lambda$, $\frac{1}{f}$ is a non-unit in $W$, and $f$ is a non-unit in $W'$.
    \end{enumerate}
\end{lemma}

\begin{proof}
Write $r = p / q$ in lowest terms, and denote $t = y^q / x^p$, so that $t$ is residually transcendental for $V_r$ and $V_r / M_r = k (\overline{t})$.

To find an element $f$ as in Item~\ref{LeftElement}, 
use the fact that $W = W_{\phi, r}$ is the pullback associated to some monic irreducible polynomial $\phi (\overline{t}) \in k [\overline{t}]$ with nonzero constant term as in Section~\ref{SectionKCR}.
Let $\Phi (u) \in D [u]$ be a monic lifting of this polynomial,
    $$\Phi (u) = u^n + a_1 u^{n-1} + \cdots + a_{n-1} u + a_n,$$  
whose nonzero coefficients $a_i$ are each units of $D$, with $a_n \ne 0$, and consider the element 
    $$\Phi (t) = t^n + a_1 t^{n-1} + \cdots + a_{n-1}t+ a_n.$$
Then for any $\lambda$,
    $$v_\lambda ( \Phi (t) ) = \min \{ n v_\lambda (t), 0 \}.$$
If $\lambda < r$, then $v_\lambda (t) < 0$ and thus $v_\lambda ( \Phi (t) ) = n v_\lambda (t) < 0$, and if $\lambda \ge r$, then $v_\lambda (t) \ge 0$ and thus $v_\lambda ( \Phi (t) ) = 0$.
Since $\Phi (t)$ is a lifting of $\phi (\overline{t})$, $W_{\phi, r}$ is the unique two-dimensional valuation overring of $D$ contained in $V_r$ which contains $\Phi (t)$ as a non-unit.
Moreover, for $\lambda > r$, since $v_\lambda (t) > 0$, then the image of $\Phi (t)$ in the residue field of $V_\lambda$ equals the image of $a_n$, which is a unit in $D$, so $\Phi (t)$ is residually algebraic.
Then the element $f = \frac{1}{\Phi (t)}$ satisfies the conditions of Item~\ref{LeftElement}.

An element $f$ satisfying the conditions in Item~\ref{RightElement} is constructed in exactly the same way, except the residue field of $V_r$ is viewed as $k (\overline{t^{-1}})$.
An element of the form $\Psi (t^{-1})$ is constructed, and $f = \frac{1}{\Psi (t^{-1})}$.

To construct an element $f$  satisfying the conditions in Item~\ref{BothElement}, construct $f_1$ and $f_2$ as in Items~\ref{LeftElement} and \ref{RightElement}, and set $f = f_1 f_2$.

Finally, to construct an element $f$ satisfying the conditions in Item~\ref{SideElement}, view the residue field of $V_r$ again as $k (\overline{t})$, and find monic polynomials $\phi (\overline{t}),$ $\phi' (\overline{t}) \in k [\overline{t}]$ such that $W = W_{\phi, r}$ and $W' = W_{\phi', r}$, respectively.
Let $\Phi (u),$ $\Phi' (u) \in D [u]$ be monic liftings whose nonzero coefficients are units of $D$.
As with Item~\ref{LeftElement}, $v_\lambda (\Phi (t)) = 0$ for $\lambda \ge r$ and $v_\lambda (\Phi (t)) = (\deg \Phi) v_\lambda (t)$ for $\lambda < r$, and similarly for $\Phi'$.
Thus, for all $\lambda$,
    $$v_\lambda ( (\Phi (t))^{\deg \Phi'}) = v_\lambda ( (\Phi' (t))^{\deg \Phi}).$$
It follows that the element
$$f = \frac{( \Phi' (t) )^{\deg \Phi}}{(\Phi (t))^{\deg \Phi'}}$$
satisfies the condition that $v_\lambda (f) = 0$ for all $\lambda$.
Since the elements $\Phi (t), \Phi' (t)$ are residually algebraic for $\lambda > r$, so is $f$.
To see that $f$ is residually algebraic for $\lambda < r$, divide the numerator and denominator of $f$ by the common leading term $t^{(\deg \Phi)(\deg \Phi')}$.
The resulting numerator and denominator each have residue $1$ in the residue field of $V_\lambda$, so the residue of $f$ is $1$.

Since the elements constructed in this lemma have non-negative $v_\lambda$ value for the interval $I$, they are almost integral over $T_I$.
\end{proof}

\section{The complete integral closure of an interval ring is a Pr\"ufer domain}

In this section, we prove that the complete integral closure ${T_I}^*$ of the interval ring $T_I$ is a Pr\"ufer domain.

Let $X$ be the collection of all valuation overrings of $T_I$.
To show that ${T_I}^*$ is Pr\"ufer, we claim that for any valuation ring $W \in X$, there exists a collection $U \subseteq K$ of almost integral elements over $T_I$ whose elements are units in $W$ and such that $W$ is the integral closure of $U^{-1} T_I [U]$ in $K$.
(This claim is trivial for $K \in X$ itself.)

This claim implies that ${T_I}^*$ is Pr\"ufer.
To see this, let $\mathfrak{p}$ be a prime ideal of ${T_I}^*$.
Then there exists a valuation overring $W$ of ${T_I}^*$, with maximal ideal $\mathfrak{m}_W$, such that $\mathfrak{m}_W \cap {T_I}^* = \mathfrak{p}$.
By the claim, there exists a set of almost integral elements $U \subseteq {T_I}^*$ such that $W$ is the integral closure of $R=U^{-1} T_I[U]$ in $K$.
Since the elements of $U$ are units in $W$, we have $R \subseteq ({T_I}^*)_{\mathfrak{p}} \subseteq W$.
The middle ring is integrally closed, so $W = ({T_I}^*)_{\mathfrak{p}}$.
We conclude that every localization of ${T_I}^*$ at a prime ideal is a valuation ring, proving that ${T_I}^*$ is Pr\"ufer.

To prove the claim, for each valuation ring $W$ of each class of valuation rings in $X$, we exhibit an explicit set $U$ of almost integral elements that are units in that valuation ring and, for every other valuation ring (except for its valuation overring in the case of a rank two valuation ring), contains an element that is a non-unit in that valuation ring.
That is, for every other valuation ring $W'$, an element of $U$ that is a unit in $W$ but not a unit in $W'$ ``eliminates'' the valuation ring $W'$.
If $U$ is a set of elements that ``eliminates'' all other valuation rings, then $W$ is the unique minimal valuation ring containing $U^{-1} T_I [U]$, so $W$ is the integral closure of $U^{-1} T_I [U]$ in $K$.

We consider each of the four classes of valuation rings, as in Section~\ref{ValuationOverringsSection}.

The first class of valuation rings are the rational monomial valuation rings $V_r$.
Take one element of the form of Lemma~\ref{ElementLemma}.\ref{BothElement} to eliminate ${W_r}^+$, ${W_r}^-$, and all valuations that are not subrings of $V_r$.
Then eliminate the rank $2$ valuation overrings with elements of the form of Lemma~\ref{ElementLemma}.\ref{SideElement}.

The second class of valuation rings are the irrational monomial valuation rings $V_r$.
Take sequences of rational numbers $\{ a_n \}_{n \ge 0}$ and $\{ b_n \}_{n \ge 0}$ that converge to $r$ from the left and right, respectively.
To eliminate the left side of the interval, for each $a_n$, take one element of the form of Lemma~\ref{ElementLemma}.\ref{LeftElement}.
Similarly, to eliminate the right side of the interval, for each $b_n$, take one element of the form of Lemma~\ref{ElementLemma}.\ref{RightElement}.

The third class of valuation rings are the ``internal'' rank two valuation rings $W$ associated to some $V_r$, with $r$ rational.
Suppose that $W = W_r^+$ and $r > a$ as in Lemma~\ref{ElementLemma}; the case $W = W_r^-$ and $r < b$ is similar.
The endpoint cases $W={W_a}^+$ and $W={W_b}^-$ are handled similarly, omitting the vacuous side of the interval.
To eliminate the left side of the interval, take any element of the form of Lemma~\ref{ElementLemma}.\ref{LeftElement}.
As in the previous case, take a sequence of rational numbers $\{ b_n \}_{n \ge 0}$ that converges to $r$ from the right.
To eliminate the right side of the interval, for each $b_n$, take one element of the form of Lemma~\ref{ElementLemma}.\ref{RightElement}.
To eliminate the remaining rank $2$ valuation overrings of $V_r$, take one element of the form of Lemma~\ref{ElementLemma}.\ref{SideElement} for each.

The fourth and final class of valuation rings are the ``external'' rank two valuation rings $W$ associated to some $V_r$, with $r$ rational.
Eliminate both sides of the interval with one element of the form of Lemma~\ref{ElementLemma}.\ref{BothElement}, and eliminate all other rank $2$ valuations with elements as in Lemma~\ref{ElementLemma}.\ref{SideElement} with one element for each.

\section{Interval rings are vacant}

Vacant domains, which were introduced in \cite[Definition 1.1]{MR2747239} and \cite[Corollary~4.10]{MR3105748}, are integrally closed domains with a unique representation as an intersection of valuation rings up to topological closure.
Aside from the class of Pr\"ufer domains \cite[Remark 2.2]{MR2747239}, few examples of vacant domains are known, and those examples that are local occur as certain pseudo-valuation domains (\cite[Section 32, Example 12]{MR1204267} and \cite[Section 4]{MR2747239}).
Interval rings provide a new and richer source of local examples.

Formally, for an integrally closed domain $T$, a collection $X$ of valuation overrings is a \textit{defining family} for $T$ if $T = \bigcap_{V \in X} V$.
The ring $T$ is a \textit{vacant domain} if $T$ has a unique defining family up to inverse closure \cite[Corollary 4.10]{MR3105748}.
That is, for any defining family $X$ of $T$, every valuation overring of $T$ is either in $X$, is a limit point of $X$ in the patch topology, or is an overring of a valuation ring in one of these two classes.

In particular, if $T$ has a defining family that is contained in every other defining family and is dense in the inverse topology, then $T$ is vacant.
We show that interval rings possess this property by giving the explicit minimal defining family.

Consider the interval $I = [a, b]$ with $a < b$, and let $T_I$ be the corresponding interval ring.
We claim that the set $X$ of ``external'' rank $2$ valuation overrings of $T_I$ is the minimal defining family.

Each rational valuation ring $V_r$ is a patch limit point of any infinite set of rank $2$ valuation rings it contains.

Each irrational valuation ring is a limit point of the rational valuation rings. 
As in Section~\ref{SectionRealTopology}, if $\{ a_n \}$ is a sequence of rational numbers converging to the irrational number $\lambda$, then $\{ V_{a_n} \}$ converges to $V_\lambda$ in the patch topology.

The ``internal'' rank $2$ valuation rings are limit points of the rational valuation rings, as in Section~\ref{SectionInternal}.

Let $W \in X$ be an ``external'' rank $2$ valuation ring.
By \ref{ElementLemma}.\ref{BothElement}, there exists an element $f \in K$ such that $f \notin W$, but $f \in V$ for all $V \in \Val (T_I) \setminus \{ W \}$.
Therefore the set $\Val (T_I) \setminus \{ W \}$ is closed in the patch topology, so $W$ is not a limit point in any defining family.

Thus $X$ is a defining family, and every defining family must include every valuation ring of $X$, so we conclude $X$ is the unique minimal defining family and $T_I$ is vacant.

\section{Other Real Intervals}

So far, we have considered closed intervals $I = [a, b]$ with $0 < a < b < \infty$.
In this section, we will consider half-open and open intervals, and real intervals whose left endpoint is $0$ or whose right endpoint is $\infty$.

Let $I = (a, b]$ with $0 < a < b < \infty$ and let $\overline{I} = [a, b]$, so that $T_{\overline{I}} \subseteq T_I$.
We have previously shown that $T_{\overline{I}}$ is the intersection of its Krull-constructed rings.
If $a$ were irrational, the intervals $(a, b]$ and $[a, b]$ induce the same set of Krull-constructed rings, implying that $T_I = T_{\overline{I}}$.
Thus we need only consider the case where $a$ is rational.

We show the ring $T_I$ shares many properties with interval rings defined by closed intervals:
\begin{enumerate}
    \item $T_I$ is local with maximal ideal $\displaystyle M_I = \bigcap_{r \in I} M_r$.
    \item $T_I$ is the intersection of the Krull-constructed rings associated to the rational numbers in $I$.
    \item $T_I$ is an almost integral extension of $T_{\overline{I}}$, and thus is vacant and has the same complete integral closure.
    \item $T_I$ is two-dimensional, its unique height one prime is the center of $V_a$, and $V_a$ is the localization of $T_I$ at its center.
\end{enumerate}

Write $I = (a, b]$ as an ascending union of closed intervals $I_n = [a_n, b]$, so that $T_I$ is a descending intersection of local rings $T_{I_n}$ with a descending sequence of maximal ideals $M_{I_n}$.
That is,
\[ I = \bigcup_{n \ge 0} I_n, \quad T_I = \bigcap_{n \ge 0} T_{I_n}, \quad M_{I} = \bigcap_{n \ge 0} M_{I_n},\]
so that $T_I$ is local with maximal ideal $M_I$.

We claim that the valuation rings containing $T_I$ are precisely the monomial valuation rings associated to the interval $I$, their rank $2$ valuation subrings, and in addition, the valuation rings $V_a$ and ${W_a}^+$.
To see this, each of the valuation rings associated to the interval $I$ is a valuation overring of some $T_{I_n}$, and so is a valuation overring of $T_I$.
Since $T_{\overline{I}} \subseteq T_I$, the only valuation rings remaining to consider are $V_a$ and its rank $2$ valuation subrings.

Let $f \in T_I$ and consider its evaluation function $\Eval_f : [a, b] \rightarrow \mathbb{R}$.
Since $f \in T_I$, $\Eval_f (\lambda) \ge 0$ on $(a, b]$, so $f \in {W_a}^+$ as in the definition of ${W_a}^+$.
Thus $T_I \subseteq {W_a}^+$, and since ${W_a}^+ \subseteq V_a$, also $T_I \subseteq V_a$.
Furthermore, if $f \in M_{I}$, then $\Eval_f (\lambda) > 0$ on $(a, b]$, hence $f \in {M_a}^+$.
Thus ${W_a}^+$ dominates $T_I$.

Finally, we show that none of the other rank $2$ valuation subrings of $V_a$ are overrings of $T_I$.
If $W$ is any such valuation subring other than ${W_a}^{-}$, then any elements produced by Lemma~\ref{ElementLemma} with $r = a$ and the rank $2$ valuation $W$ are in $T_I$ and not in $W$, showing that $W$ is not an overring of $T_I$.
For ${W_a}^{-}$, the reciprocal of an element produced by Lemma~\ref{ElementLemma}.\ref{LeftElement} (with $r = a$ and $W$ any rank $2$ valuation subring of $V_a$) is in $T_I$ but not in ${W_a}^{-}$, showing that ${W_a}^{-}$ is not an overring of $T_I$.

For a real interval $I = (a, b)$, with $0 < a < b < \infty$ and $a, b$ rational, the same arguments and properties hold, except:
\begin{enumerate}
    \item[(4')] $T_I$ is two-dimensional, and it has precisely two height one primes, the centers of $V_a$ and $V_b$, and $V_a$ and $V_b$ are the localizations of $T_I$ at their centers.
\end{enumerate}

To consider intervals whose endpoints are zero or infinity, denote $V_x = D_{(x D)}$ and $V_y = D_{(y D)}$, the $x$-adic and $y$-adic valuation rings, and let $W_x$ and $W_y$ denote their unique rank $2$ valuation subrings dominating $D$.
For an interval $[0, b]$ (with $b < \infty$), or an interval $[a, \infty]$ (with $a > 0$), or the interval $[0, \infty]$, we extend the definition of interval rings using $W_x$ or $W_y$ in place of the Krull-constructed ring for $0$ or $\infty$, respectively.

Consider the interval $I = (0, b]$, with $b < \infty$, and $\overline{I} = [0, b]$.
As before, write the interval $I$ as an ascending union of closed intervals $[a_n, b]$, so that $T_I$ is the descending intersection of local rings $T_{I_n}$.
By the same argument using evaluation functions (as for the interval $(a, b]$), $W_x$ dominates $T_I$.
Since $T_{\overline{I}}$ is defined to be the intersection of $T_{I}$ with $W_x$, it follows that $T_{\overline{I}} = T_I$.

Arguments similar to those in Section~\ref{ValuationOverringsSection} show that the only valuation overrings of $T_{\overline{I}}$ are those used in its construction.
For rank $1$ valuations $w$ such that $w (y) \in I$ or $w (y) > b$, the exact same argument applies.
The only rank $1$ valuation $w$ with $w (y) = 0$ is the valuation for $V_x$, and its only rank $2$ valuation, $W_x$, is used in the construction of $T_{\overline{I}}$.

We conclude that $T_I = T_{\overline{I}}$ satisfies the same properties (1)-(4).

For an interval $[a, \infty]$, we may swap the variables $x$ and $y$ and instead consider the equivalent ring for the interval $[0, \frac{1}{a}]$, so it also satisfies the same properties.

In summary, using an open rational endpoint, or an endpoint of $0$ or $\infty$ (whether open or closed), keeps the interval ring local, but adds a height one non-maximal prime ideal.
This non-maximal prime ideal is the center of the monomial valuation ring corresponding to the endpoint.
Using two such endpoints adds two such height one prime ideals.

\section{Connections with more ordinary classes of rings}

We know that an interval ring (built using a nontrivial closed interval) is a vacant one-dimensional integrally closed local ring which is neither a valuation ring nor completely integrally closed.
It follows easily that an interval ring is not Noetherian, not Pr\"ufer, and not Krull.
It is also not coherent, or even a finite conductor domain, since a one-dimensional local integrally closed finite conductor domain is a valuation domain \cite[Theorem 1]{MR304371}.

Interval rings satisfy ACCP, but are not Mori domains.
That is, every ascending chain of principal ideals eventually stabilizes, but there exists a non-stabilizing chain of divisorial ideals.
Indeed, since $T_I$ is a local ring that's dominated by a DVR, it satisfies ACCP.
To see this, let $v$ be an integer-valued valuation of a DVR dominating $T_I$.
If $a T_I \subsetneq b T_I$, then $\frac{a}{b} \in M_{T_I}$, so $v (a) > v (b)$.
Thus no strictly ascending chain of principal ideals starting from $a T_I$ can exceed $v (a)$ in length.

However, $T_I$ is not a Mori domain.
To see this, we show that ${T_I}^*$ does not satisfy ACCP.
We use the fact that, as in Section~\ref{SectionEvaluation}, for any rational number $p / q \in [a,b]$, written in lowest terms, there exists an element $f \in {T_I}^*$ such that $\max \Eval_f ([a, b]) = 1 / q.$

Take a sequence of rational numbers $r_i = p_i / q_i \in I$ such that $$\sum_{i=1}^{\infty} 1 / q_i < 1,$$ and take the corresponding sequence of elements $f_i \in {T_I}^*$ with $$\max \Eval_{f_i} ([a, b]) = 1 / q_i.$$
Then,
    $$\frac{x}{f_1} {T_I}^* \subsetneq \frac{x}{f_1 f_2} {T_I}^* \subsetneq \cdots \subsetneq \frac{x}{f_1 \cdots f_n} {T_I}^* \subsetneq \cdots$$
is a strictly ascending chain of principal ideals in ${T_I}^*$.
To prove that this is a chain of divisorial ideals in $T_I$, it suffices to show that ${T_I}^*$ is a divisorial fractional ideal of $T_I$, and to prove this, it is enough to show that ${T_I}^* = (T_I :_{K} M_I)$.
Since $M_I$ is the intersection of the maximal ideals of the rings $V_r$, for $r \in I$, and ${T_I}^*$ is the intersection of these rings, we have ${T_I}^* \subseteq (M_I:_K M_I) \subseteq (T_I :_{K} M_I)$. Thus $M_I$ is not invertible since $T_I \ne {T_I}^* \subseteq (M_I:_K M_I)$.
Since $M_I$ is a non-invertible maximal ideal of $T_I$, it follows that $M_I(T_I :_{K} M_I) = M_I$, so that $(T_I :_{K} M_I) \subseteq (M_I:_K M_I) \subseteq (T_I :_{K} M_I)$, proving that $(T_I :_K M_I) = (M_I :_K M_I)$.
Since $(M_I:_K M_I) \subseteq {T_I}^*$ from the fact that ${T_I}^*$ is completely integrally closed, and since ${T_I}^* \subseteq (M_I :_K M_I)$ was established above, we conclude that ${T_I}^* = (M_I :_K M_I)$, and hence ${T_I}^* = (T_I :_K M_I)$.
Thus ${T_I}^*$ is a divisorial fractional ideal of $T_I$, so $T_I$ is not a Mori domain.

\section{Acknowledgments}

We thank the referees for valuable suggestions that improved the paper, including the argument in Section~\ref{ValuationOverringsSection} that eliminates the requirement that the residue field have more than two elements.

\medskip 

\printbibliography

@article {MR3105748,
    AUTHOR = {Finocchiaro, Carmelo A. and Fontana, Marco and Loper, K. Alan},
     TITLE = {The constructible topology on spaces of valuation domains},
   JOURNAL = {Trans. Amer. Math. Soc.},
  FJOURNAL = {Transactions of the American Mathematical Society},
    VOLUME = {365},
      YEAR = {2013},
    NUMBER = {12},
     PAGES = {6199--6216},
      ISSN = {0002-9947,1088-6850},
   MRCLASS = {13F30 (13A18)},
  MRNUMBER = {3105748},
MRREVIEWER = {Ry\={u}ki\ Matsuda},
       DOI = {10.1090/S0002-9947-2013-05741-8},
       URL = {https://doi.org/10.1090/S0002-9947-2013-05741-8},
}

@article {MR2747239,
    AUTHOR = {Fabbri, Alice},
     TITLE = {Integral domains having a unique {K}ronecker function ring},
   JOURNAL = {J. Pure Appl. Algebra},
  FJOURNAL = {Journal of Pure and Applied Algebra},
    VOLUME = {215},
      YEAR = {2011},
    NUMBER = {5},
     PAGES = {1069--1084},
      ISSN = {0022-4049,1873-1376},
   MRCLASS = {13A15 (13F05)},
  MRNUMBER = {2747239},
MRREVIEWER = {David\ E.\ Rush},
       DOI = {10.1016/j.jpaa.2010.07.012},
       URL = {https://doi.org/10.1016/j.jpaa.2010.07.012},
}

@article {MR2289617,
    AUTHOR = {Granja, A.},
     TITLE = {The valuative tree of a two-dimensional regular local ring},
   JOURNAL = {Math. Res. Lett.},
  FJOURNAL = {Mathematical Research Letters},
    VOLUME = {14},
      YEAR = {2007},
    NUMBER = {1},
     PAGES = {19--34},
      ISSN = {1073-2780},
   MRCLASS = {13A18 (13H05)},
  MRNUMBER = {2289617},
MRREVIEWER = {Paolo\ Zanardo},
       DOI = {10.4310/MRL.2007.v14.n1.a2},
       URL = {https://doi.org/10.4310/MRL.2007.v14.n1.a2},
}

@article {MR251026,
    AUTHOR = {Hochster, M.},
     TITLE = {Prime ideal structure in commutative rings},
   JOURNAL = {Trans. Amer. Math. Soc.},
  FJOURNAL = {Transactions of the American Mathematical Society},
    VOLUME = {142},
      YEAR = {1969},
     PAGES = {43--60},
      ISSN = {0002-9947,1088-6850},
   MRCLASS = {13.20 (14.00)},
  MRNUMBER = {251026},
MRREVIEWER = {S.\ S.\ Page},
       DOI = {10.2307/1995344},
       URL = {https://doi.org/10.2307/1995344},
}

@incollection {MR4394424,
    AUTHOR = {Olberding, Bruce},
     TITLE = {The {Z}ariski-{R}iemann space of valuation rings},
 BOOKTITLE = {Commutative algebra},
     PAGES = {639--667},
 PUBLISHER = {Springer, Cham},
      YEAR = {2021},
      ISBN = {978-3-030-89693-5; 978-3-030-89694-2},
   MRCLASS = {14A05 (13A18)},
  MRNUMBER = {4394424},
       DOI = {10.1007/978-3-030-89694-2\_21},
       URL = {https://doi.org/10.1007/978-3-030-89694-2_21},
}

@article {MR54571,
    AUTHOR = {Seidenberg, A.},
     TITLE = {A note on the dimension theory of rings},
   JOURNAL = {Pacific J. Math.},
  FJOURNAL = {Pacific Journal of Mathematics},
    VOLUME = {3},
      YEAR = {1953},
     PAGES = {505--512},
      ISSN = {0030-8730},
   MRCLASS = {09.1X},
  MRNUMBER = {54571},
MRREVIEWER = {I. S. Cohen},
       URL = {https://projecteuclid.org/euclid.pjm/1103051409},
}

@book {MR1204267,
    AUTHOR = {Gilmer, Robert},
     TITLE = {Multiplicative ideal theory},
    SERIES = {Queen's Papers in Pure and Applied Mathematics},
    VOLUME = {90},
      NOTE = {Corrected reprint of the 1972 edition},
 PUBLISHER = {Queen's University, Kingston, ON},
      YEAR = {1992},
     PAGES = {xii+609},
   MRCLASS = {13-02 (13A15)},
  MRNUMBER = {1204267},
}

@unpublished {2406.10966,
    AUTHOR      = {Heinzer, William and Loper, K. Alan and Olberding, Bruce and
                   Toeniskoetter, Matthew},
    TITLE       = {Connectedness and integrally closed local overrings of two-dimensional regular local rings},
    EPRINT      = {2406.10966},
    EPRINTTYPE  = {arxiv},
    EPRINTCLASS = {math.AC},
    NOTE        = {Submitted},
    YEAR        = {2024}
}

@article {MR4719886,
    AUTHOR = {Heinzer, William and Alan Loper, K. and Olberding, Bruce and
              Toeniskoetter, Matthew},
     TITLE = {A connectedness theorem for spaces of valuation rings},
   JOURNAL = {J. Algebra},
  FJOURNAL = {Journal of Algebra},
    VOLUME = {647},
      YEAR = {2024},
     PAGES = {844--857},
      ISSN = {0021-8693,1090-266X},
   MRCLASS = {13A18},
  MRNUMBER = {4719886},
       DOI = {10.1016/j.jalgebra.2024.02.032},
       URL = {https://doi.org/10.1016/j.jalgebra.2024.02.032},
}

@book {MR1011461,
    AUTHOR = {Matsumura, Hideyuki},
     TITLE = {Commutative ring theory},
    SERIES = {Cambridge Studies in Advanced Mathematics},
    VOLUME = {8},
   EDITION = {Second},
      NOTE = {Translated from the Japanese by M. Reid},
 PUBLISHER = {Cambridge University Press, Cambridge},
      YEAR = {1989},
     PAGES = {xiv+320},
      ISBN = {0-521-36764-6},
   MRCLASS = {13-01},
  MRNUMBER = {1011461},
}

@book {MR242802,
    AUTHOR = {Atiyah, M. F. and Macdonald, I. G.},
     TITLE = {Introduction to commutative algebra},
 PUBLISHER = {Addison-Wesley Publishing Co., Reading, Mass.-London-Don
              Mills, Ont.},
      YEAR = {1969},
     PAGES = {ix+128},
   MRCLASS = {13.00},
  MRNUMBER = {242802},
MRREVIEWER = {Johnny\ A.\ Johnson},
}

@article {MR220714,
    AUTHOR = {Gilmer, Jr., Robert W. and Heinzer, William J.},
     TITLE = {On the complete integral closure of an integral domain},
   JOURNAL = {J. Austral. Math. Soc.},
  FJOURNAL = {J. Austral. Math. Soc.},
    VOLUME = {6},
      YEAR = {1966},
     PAGES = {351--361},
   MRCLASS = {13.15},
  MRNUMBER = {220714},
MRREVIEWER = {E.\ H.\ Batho},
}

@article {MR304371,
    AUTHOR = {McAdam, Stephen},
     TITLE = {Two conductor theorems},
   JOURNAL = {J. Algebra},
  FJOURNAL = {Journal of Algebra},
    VOLUME = {23},
      YEAR = {1972},
     PAGES = {239--240},
      ISSN = {0021-8693},
   MRCLASS = {13G05},
  MRNUMBER = {304371},
MRREVIEWER = {J.\ T.\ Arnold},
       DOI = {10.1016/0021-8693(72)90128-7},
       URL = {https://doi.org/10.1016/0021-8693(72)90128-7},
}

@article {MR1545646,
    AUTHOR = {Krull, Wolfgang},
     TITLE = {Beitr\"age zur {A}rithmetik kommutativer
              {I}ntegrit\"atsbereiche},
   JOURNAL = {Math. Z.},
  FJOURNAL = {Mathematische Zeitschrift},
    VOLUME = {41},
      YEAR = {1936},
    NUMBER = {1},
     PAGES = {665--679},
      ISSN = {0025-5874,1432-1823},
   MRCLASS = {99-04},
  MRNUMBER = {1545646},
       DOI = {10.1007/BF01180447},
       URL = {https://doi.org/10.1007/BF01180447},
}
 
\end{document}